\documentclass[a4paper,11pt]{article}
\usepackage{latexsym}
\usepackage{amssymb}
\usepackage{bbm}
\usepackage{enumitem}
\usepackage{amsfonts}
\usepackage{amsmath}
\usepackage{amsthm}
\usepackage{cases}
\usepackage[pdftex]{graphicx} 
\usepackage[paper=a4paper,left=30mm,right=20mm,top=25mm,bottom=30mm]{geometry}

\newenvironment{@abssec}[1]{%
    \if@twocolumn

      \section*{#1}%
    \else

      \vspace{.05in}\footnotesize
      \parindent .2in
 {\upshape\bfseries #1. }\ignorespaces
    \fi}

    {\if@twocolumn\else\par\vspace{.1in}\fi}

\newenvironment{keywords}{\begin{@abssec}{\keywordsname}}{\end{@abssec}}

\newenvironment{AMS}{\begin{@abssec}{\AMSname}}{\end{@abssec}}

\newcommand\keywordsname{Key words}
\newcommand\AMSname{AMS subject classifications}
\newcommand\AMname{AMS subject classification}
\newtheorem{theorem}{Theorem}
 \newtheorem{lemma}[theorem]{Lemma}

\title{A construction of patterns with many critical points on topological tori
\footnotetext[3]{This research was partially supported by the Grant-in-Aid
for Scientific Research (B) ($\sharp$ 18H01126) of
Japan Society for the Promotion of Science.}\footnotemark[3]}

\author{Putri Zahra Kamalia\footnote{Corresponding author.} \footnote{Research Center for Pure and Applied Mathematics, Graduate School of Information Sciences, Tohoku
University, Sendai, 980-8579, Japan ({\tt putrizahrakamalia@gmail.com}${}^*$, {\tt  sigersak@tohoku.ac.jp}).}\  \  and \  Shigeru Sakaguchi\footnotemark[2]} 

\date{}
\begin{document}

\maketitle

\begin{abstract}
We consider reaction-diffusion equations on closed surfaces in $\mathbb R^3$ having genus $1$. Stable nonconstant stationary solutions are often called patterns.  The purpose of this paper is to construct closed surfaces together with patterns having as many critical points as one wants.
\end{abstract}


\section{Introduction}
In this paper, we deal with reaction-diffusion problems on closed surfaces properly embedded in $\mathbb R^3$. We will address the existence of stable nonconstant stationary solutions having  critical points as many as possible.  For a closed surface  $M$  properly embedded in $\mathbb R^3$ with a Riemannian metric $g$, we consider the following reaction-diffusion problem  for $u =u(x,t)$ on $M$ 
\begin{equation}\label{eq1}
\partial_t u= \Delta_g u + f(u)\quad \mbox{ in } M \times (0,\infty),
\end{equation}
where $f \in C^1(\mathbb{R})$  is a function of $u$  and $\Delta_g$  denotes  the Laplace-Beltrami operator on $M$,
\begin{equation}\label{gradu}
\Delta_{g} u = \textrm{div} (\nabla_{g} u)= \sum_{i=1}^2 \frac{1}{\sqrt{|g|}} \frac{\partial}{\partial x^i}  \left(\sqrt{|g|} (\nabla_{g} u)^i \right).
\end{equation}
Stable nonconstant stationary solutions of \eqref{eq1} are often called patterns. The stability of stationary solutions  of \eqref{eq1} is discussed in the sense of Lyapunov. We say that a stationary solution U  is stable if, for each $\epsilon > 0$, there exists a $\delta > 0$ such that, for every initial data $u_0$ with
$\Vert {u_0-U}\Vert_{\infty} < \delta$ we have
$$
\Vert u(\cdot,t)-U\Vert_{\infty} < \epsilon\ \mbox{ for every } t > 0.
$$

The existence and nonexistence of patterns on compact $d$-dimensional Riemannian manifolds have been studied in \cite{bandel, farina, jimbo, nascimento,  rubinstein}. Among these, in \cite[Theorem 2]{jimbo} Jimbo has introduced manifolds and nonlinearities $f$ having complex patterns whose construction is analogous to that in \cite{matano}.  To be precise, Jimbo introduced a finite number of compact Riemannian manifolds connected by thin channels, and he used the singular perturbation method as all the thin channels shrink to line segments.  The resulting patterns are close to each given constant on each compact Riemannian manifold. 

In this paper, by using the patterns constructed in \cite{bandel} together with the implicit function theorem,  we introduce a different and simple way to construct closed surfaces having genus $1$ together with patterns having as many critical points as one wants.  Both our closed surfaces and our patterns are different from those obtained in  \cite[Theorem 2]{jimbo}.  We will explain their construction  below.

For a surface of revolution $D$ in $\mathbb R^3$ with  boundary $\partial D$, 
there are several studies on the existence of patterns of \eqref{eq1} where $M$ is replaced with $D$ under various boundary conditions on $\partial D$.  In  \cite{bandel} Bandle, Punzo and Tesei constructed a class of surfaces of revolution $D$ with non-empty boundary and nonlinear terms $f$ having patterns with the Neumann boundary condition by solving some ordinary differential equations with the aid of an idea introduced by Yanagida in \cite{yanagida} to  construct the nonlinear terms $f$.  Recently,  in \cite{sonego} Sonego obtained patterns with the Dirichlet boundary condition on $\partial D$.  On the other hand, in \cite{punzo} Punzo  considered the case where $\partial D = \emptyset$, that is,  $D$ is a closed surface of genus $0$,  and shows the existence of patterns on $D$.

 On the surface of revolution $D$, the reaction-diffusion problem with the Neumann boundary condition reads as
\begin{equation}\label{rdn}
\begin{cases}
\partial_t u= \Delta_g u + f(u) &\mbox{ in } D \times (0,\infty),\\
\dfrac{\partial u}{\partial \nu}=0 &\mbox{ on } \partial D \times (0,\infty),
\end{cases}
\end{equation}
where $\nu$ denotes the outward unit normal vector to $\partial D$. We prove in this paper that the patterns of problem \eqref{rdn} still exist even if the domain $D$ dealt with in \cite{bandel} is slightly perturbed. The perturbation is done by deforming its axis of rotation into a circular arc. We start with the patterns constructed in \cite{bandel}  and examine the existence of patterns with the Neumann boundary condition on the perturbed domain by using the implicit function theorem. Furthermore, we manage to construct patterns together with surfaces of genus $1$  by attaching a finite number of copies of the perturbed domain together with the perturbed pattern. This construction comes from an idea of Below and Lubary in \cite[Theorem 4.2, p.177]{below}  which introduces patterns on a {\it disjoint union } of graphs. As a consequence of this construction, the Neumann boundary condition on the boundary of each copy of the perturbed domain yields the critical points of the resulting patterns. We summarize our result in the following theorem.
\begin{theorem}\label{mt}
	There exist a nonlinearity $f$ and a number $N \in \mathbb{N}$ such that, for each $n \geq N$, a closed surface $M$ of genus $1$ properly embedded in $\mathbb{R}^3$ together with a pattern $U$  of \eqref{eq1} on $M$, is constructed in such a way that $U$ has at least $4n$ critical points.
\end{theorem}

This paper is organized as follows. In section 2, we introduce the patterns on surfaces of revolution $D$ constructed in \cite{bandel}. 
Then, we reparameterize the domain in order to adjust it to our problem.  In section 3 we give a detail of construction of surfaces of genus  $1$. Section 4 is devoted to the proof of Theorem \ref{mt}.

\setcounter{equation}{0}
\setcounter{theorem}{0}
\section{Preliminary}

As in \cite[subsection 2.2 (p. 36) and Theorem  3.6 (p. 40)]{bandel}, let $C$ be a regular plane curve parameterized by 
\begin{equation*}
\begin{cases}
x_1=\psi(r),\\
x_2=0,  \qquad \left( r \in [0,L] \right) \\ 
x_3=\chi(r),
\end{cases}
\end{equation*}
where $\psi, \chi \in C^3([0,L])$, $\psi >0$ in $[0,L]$ and $(\psi^\prime)^2+(\chi^\prime)^2 =1$.
Then by revolving  the curve $C$ about the $x_3$-axis, it admits a surface of revolution $D$ in $\mathbb{R}^3$ parameterized by
\begin{equation}
\begin{cases}
x_1=\psi(r) \cos \theta, \\
x_2=\psi(r) \sin \theta,   \qquad \left( (r,\theta) \in [0,L]\times (0,2\pi] \right)\\ 
x_3=\chi(r)
\end{cases}
\end{equation} 
with local coordinates $x^1=r, x^2=\theta$. For $r \in [0,L]$ , set 
\begin{equation}
C_r= \lbrace  (\psi(r) \cos \theta, \psi(r) \sin \theta, \chi(r)): \theta \in (0,2\pi]\rbrace.
\end{equation}
Then, $\partial D= C_0 \cup C_L$.

The Riemannian metric $g=(g_{ij})$ on $D$ is given by
\begin{equation*}
(g_{ij})_{i,j=1,2}=
\begin{pmatrix}
1 & 0 \\
0 & \psi^2(r)
\end{pmatrix},
\end{equation*}
the area element on $D$ is $dV_g=\sqrt{|g|} dr d\theta=\psi dr d\theta$, and
the Riemannian gradient $\nabla_g u$ of $u$ with respect to $g$ on $D$ is  given by
 \begin{equation*}
 \nabla_{g} u =
 \begin{pmatrix}
 \partial_r u \\[0.25cm]
 \dfrac{1}{\psi^2(r)} \partial_\theta u
 \end{pmatrix}.
 \end{equation*}
The study of existence and nonexistence of patterns of the reaction-diffusion problem \eqref{rdn} has already been done by \cite{bandel}. 
Indeed, it is shown in \cite[Theorem 4.1, p. 41]{bandel} that the existence of an interior point $\hat{R} \in (0,L)$ with $\left(\frac{\psi^\prime}{\psi}\right)^\prime (\hat{R}) > 0$ leads to the existence of a nonlinear term $f\in C^1(\mathbb{R})$ admitting patterns. First, they consider the Cauchy problem that fulfills the above criteria to construct $f$ as a $C^1$-function. The nonlinear term $f=f(Z)$ is defined by \cite[(4.10), p. 43]{bandel} in such way, that for any $r \in(0,L)$, 
\begin{equation*}
f[Z(r)]= -\dfrac{(\psi Z^\prime)^\prime}{\psi}(r),
\end{equation*}
where $Z(r)$ is a stationary nonconstant solution of problem \eqref{rdn} \cite[Lemma 4.4 and (4.13), p. 43]{bandel}. Moreover, $Z(r)$ is positive and strictly increasing in $(0,L)$. Then, the stability of $Z(r)$ is established by \cite[Lemma 4.2, p. 41]{bandel}. Hence, problem \eqref{rdn} with the nonlinear term $f$ admits a pattern $Z(r)$  on $D$.

In particular we choose D to have
\begin{align}
\chi^\prime(r) > 0 &\textrm{ for } r\in [0,L],  \label{as1}\\
\dfrac{d^{i} \psi}{dr^i}(r)=0 &\textrm{ for } r\in \{ 0,L \} \textrm{ and } i=1,3. \label{as2}
\end{align}

Let us introduce the change of variables
\begin{align*}
\Psi (s)=\psi (\chi^{-1}(s)), \quad s =\chi (r),
\end{align*}
where
\begin{equation*}
r=\int\limits_{0}^{s} \sqrt{1 + (\Psi^\prime (t))^2 }dt, \quad  s\in[0,l].
\end{equation*}
Notice in particular that 
\begin{equation}
\label{change of variable from r to s}
\chi^\prime(r) = \frac{ds}{dr} = \frac 1{\sqrt{1 +  (\Psi^\prime(s))^2}}.
\end{equation}
Then, we reparameterize $D$ as
\begin{equation} \label{surev}
\begin{cases}
x_1=\Psi(s) \cos \theta, \\
x_2=\Psi(s) \sin \theta,   \hspace{1cm} ( (s, \theta) \in I:=[0,l] \times S^1)\\ 
x_3=s.
\end{cases}
\end{equation}
The corresponding new Riemannian metric $g=(g_{ij})$ on $D$ is given by
\begin{equation}\label{rmetric}
(g_{ij})_{i,j=1,2}=
\begin{pmatrix}
1+(\Psi^\prime(s))^2 & 0\\
0 & \Psi^2(s)
\end{pmatrix}.
\end{equation}
Hence, the new Riemannian gradient $\nabla_{g} u$ of $u$ with respect to $g$ is
\begin{equation}\label{rgrad}
\nabla_{g} u =
 \begin{pmatrix}
   \dfrac{1}{1+(\Psi^\prime(s))^2 } \partial_s u \\[0.5cm] 
   \dfrac{1}{\Psi^2(s)} \partial_\theta u \end{pmatrix},
\end{equation}
and the Laplace-Beltrami operator $\Delta_g$ on $D$ is given by
\begin{equation}\label{lbk}
\Delta_{g} u= \frac{1}{1+(\Psi^\prime)^2} u_{ss} + \frac{1}{\Psi^2} u_{\theta \theta} +\frac{[1+ (\Psi^\prime)^2- \Psi^{\prime\prime}\Psi]\Psi^\prime}{[1+(\Psi^\prime)^2]^2 \Psi} u_s.
\end{equation}
Under this reparametrization, stability criteria in \cite{bandel} can be rewritten as follows:
\begin{theorem}[\cite{bandel}, Theorem 4.1, p. 41] \label{main1}
	Suppose that for some $s_0 \in (0,l)$
		\begin{equation}\label{stas}
		\Psi^{\prime \prime}\Psi - \left(\Psi^\prime\right)^2 \left[1+ (\Psi^\prime)^2\right] > 0 \quad \textrm{ at }s=s_0.
		\end{equation}
		Then, there exists $f\in C^1 (\mathbb{R})$ such that problem \eqref{rdn} admits a pattern $U=U_g(s)$.
\end{theorem}
If we set $Z(r)=U_g(\chi(r))$, the function $f\in C^1(\mathbb{R})$ of problem \eqref{rdn} is given by
\begin{equation}\label{ff}
f[U_g(s)]= -\dfrac{\Psi^\prime U_g^\prime + \Psi U_g^{\prime\prime}}{\Psi [1+(\Psi^\prime)^2]} + \dfrac{\Psi^\prime\Psi^{\prime\prime}U_g^\prime}{[1+(\Psi^\prime)^2]^2}.
\end{equation}

\setcounter{equation}{0}
\setcounter{theorem}{0}
\section{Construction of closed surfaces having genus $1$}
Let us introduce a circular arc $C(\kappa)$ for sufficiently small $|\kappa|$.
\begin{equation}\label{ckappa}
C(\kappa) =  \left\lbrace (x_1, 0, x_3)= p(s):= \left( \frac{1}{\kappa} (1 - \cos \kappa s), 0, \frac{1}{\kappa} \sin \kappa s \right): 0\leq s \leq l \right\rbrace.
\end{equation}
The above parameterization shows that the $C(\kappa)$ is pliable by controlling the value of $\kappa$. If $\kappa=0$, we recognize that $C(\kappa)$ is congruent with the segment $[0,l]$ in the $x_3$-axis.
\begin{figure}[h]
	\centering
	\includegraphics[width=0.75\textwidth]{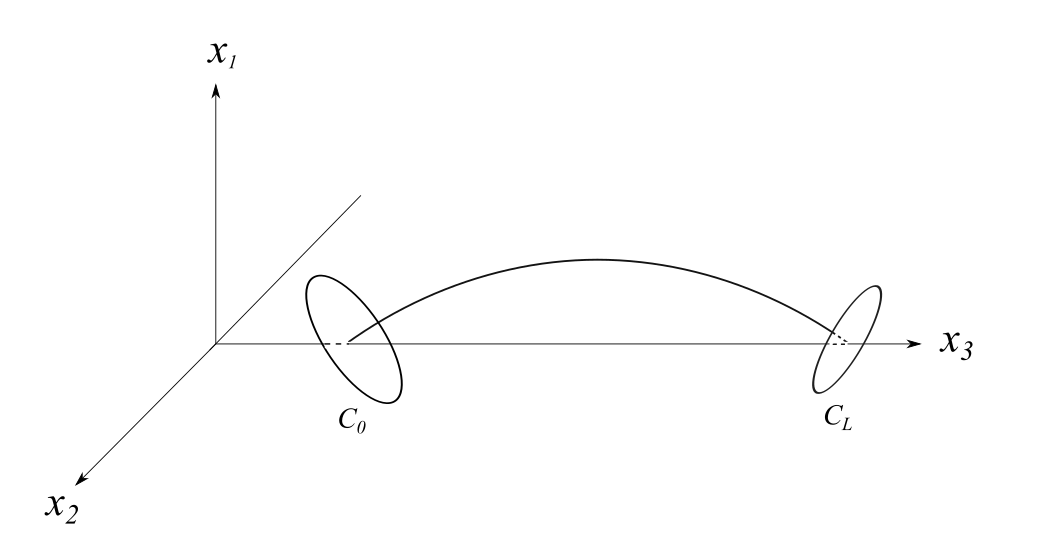}
	\caption{$C(\kappa)$ is bounded by two circles congruent with  $C_0$ and $C_L$.}
	\label{tree}
\end{figure}

For sufficiently small $|\kappa|$, let us introduce a Riemannian surface $M_{\kappa}$ as follows. While bending the line segment  $[0,l]$ into $C(\kappa)$, we preserve the value of $\Psi(s)$ by using Frenet trihedron of the curve $C(\kappa)$ (see \cite[p. 20]{do}). Let $(\mathbbm{t}(s), \mathbbm{n}(s), \mathbbm{b}(s))$ be the Frenet trihedron of $C(\kappa)$ where
\begin{equation}
\mathbbm{t}(s)=\begin{pmatrix} \sin \kappa s\\ 0\\\cos \kappa s \end{pmatrix}, \quad
\mathbbm{n}(s)=\begin{pmatrix}\cos \kappa s\\ 0\\  -\sin \kappa s \end{pmatrix}, \quad
\mathbbm{b}(s)=\begin{pmatrix} 0\\ 1\\  0 \end{pmatrix}.
\end{equation}
Then, we define a parametrization $x=x(s,\theta) \in M_\kappa$ by 
\begin{equation*}
x(s,\theta)= p(s)+ \Psi (s)\cos \theta \cdot \mathbbm{n}(s) + \Psi(s) \sin \theta \cdot \mathbbm{b}(s)\quad  \mbox{ for }\ 
 (s, \theta) \in [0,l] \times S^1,
\end{equation*} 
where $p(s)$ is given by \eqref{ckappa}.  Hence $M_{\kappa}$ is parameterized by
\begin{equation}\label{mgk}
\begin{cases}
x_1=\dfrac{1}{\kappa} (1 - \cos \kappa s)+ \Psi(s) \cos \theta \cos \kappa s, \\[0.25cm] 
x_2=\Psi(s) \sin \theta,    \hspace{5cm} \left( (s, \theta) \in [0,l] \times S^1\right) \\[0.25cm] 
x_3=\dfrac{1}{\kappa} \sin \kappa s - \Psi(s) \cos \theta \sin \kappa s.
\end{cases}
\end{equation}
\begin{figure}[h]
	\centering
	\includegraphics[width=0.35\textwidth]{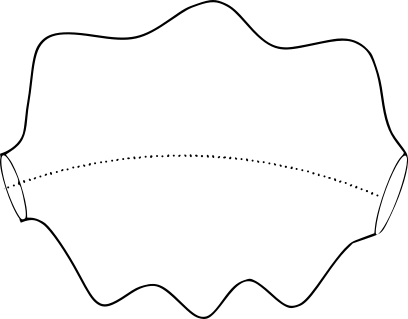}
	\caption{A modification of a surface of revolution $M_{\kappa}$ with $\kappa \neq 0$.}
	\label{Mgk}
\end{figure}
Thus, there exists $\delta_0 >0$ such that,  if $|\kappa| \leq \delta_0$, then $M_{\kappa}$ is a properly embedded Riemannian surface in $\mathbb{R}^3$  parameterized in local coordinates: $x^1=s$ and $x^2=\theta$. Let us call $C(\kappa)$ the {\it center curve} of $M_{\kappa}$.
 The corresponding Riemannian metric $g^\kappa=(g^\kappa_{ij})$ for $M_{\kappa}$ is given by
\begin{equation*}
({g^\kappa}_{ij})_{i,j=1,2}= 
\begin{pmatrix}
(\Psi^\prime(s))^2+(\kappa \Psi(s)\cos \theta-1)^2& 0\\[0.25cm]
0& \Psi^2(s)
\end{pmatrix},  
\end{equation*}
the area element on $M_{\kappa}$ is  $dV_{g^\kappa}=\sqrt{|g^\kappa|} dr d\theta$, 
and the Riemannian gradient $\nabla_{g^\kappa} u$ of $u$ with respect to $g^\kappa$ is 
\begin{equation*}
\nabla_{g^\kappa} u = \begin{pmatrix} \dfrac{1}{(\Psi^\prime(s))^2+(\kappa \Psi(s)\cos \theta-1)^2} \partial_s u \\[0.5cm] 
 \dfrac{1}{\Psi^2(s)} \partial_\theta u
\end{pmatrix}.
\end{equation*}
The Laplace-Beltrami operator $\Delta_{g^\kappa}$ on $M_\kappa$ is given by
\begin{equation}\label{Laplacian perturbed}
\Delta_{g^\kappa} u= \frac 1{\Phi^2}u_{ss} + \frac 1{\Psi^2}u_{\theta\theta} + \frac{\Psi^\prime\Phi-\Phi_s\Psi}{\Phi^3\Psi}u_s+\frac{\Phi_\theta}{\Phi\Psi^2}u_\theta,
\end{equation}
where we put $\Phi = \Phi(s,\theta) = \sqrt{(\Psi^\prime(s))^2+(\kappa \Psi(s)\cos \theta-1)^2}.$
Note that $g^0=g$.  We observe  from \eqref{as1}, \eqref{as2} and \eqref{mgk} that if $|\kappa| < \delta_0$, then
 \begin{enumerate}[label=(\arabic*)]
	\item\label{i1} The boundary of $M_{\kappa}$  consists of two circles congruent with $C_0$ and $C_L$,
	\item \label{i2} $\Psi(0)>0$, $\Psi(l)>0$,
	\item\label{i3}$\dfrac{d^i\Psi}{ds^i}(s)=0, i=1,3 \mbox{ at } s \in\left\lbrace 0, l\right\rbrace $.
\end{enumerate}
Moreover, by choosing $\delta_0$ smaller if necessary, we are allowed to form a closed surface by attaching a finite number of copies of $M_{\kappa}$, as follows: 
\begin{enumerate}[label=Step \arabic*.]
	\item \label{s0} Choose a component congruent with $C_0$ of the boundary $\partial M_{\kappa}$.
	\item Take two copies $M^1$ and $M^2$ of $M_{\kappa}$.
	\item \label{s3} Attach  the component of the boundary, which is congruent with $C_0$,  of $M^1$ to that of $M^2$ in such way that a new surface 
	$
	M_* = M^1 \cup  M^2 
$ is symmetric with respect to the hyperplane containing the circle congruent with $C_0$. The boundary of $M_*$ consists of two components both of which are congruent with $C_L$ and the center curve of $M_*$ is just a circular arc. 
	$C^3$ smoothness of  $M_*$ is guaranteed by \ref{i3}.
	\begin{figure}[h!]
		\centering
		\includegraphics[width=0.5\textwidth]{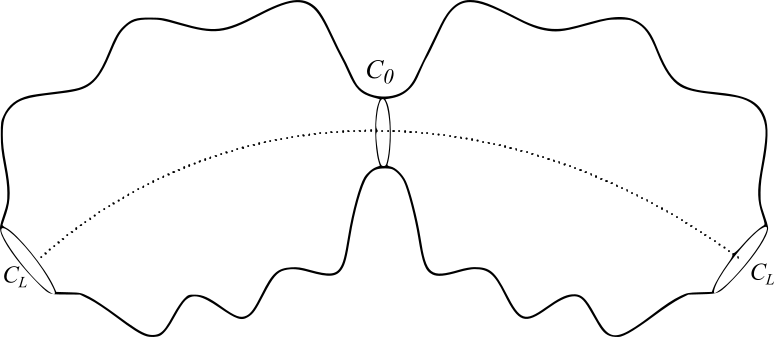}
\caption{A new surface $M_*$ composed of two $M_{\kappa}$.} 
	\end{figure}
	\item \label{s4} Take two copies $M^{1}_*$ and $M^{2}_*$ of $M_*$ and attach one of the boundary components  of $M^{1}_*$ to that of $M^{2}_*$ in such way that the center curve of a new surface $M_{2*}= M_*^1 \cup M_*^2$ is just a circular arc.
	$C^3$ smoothness of  $M_{2*}$ is guaranteed by \ref{i3}.
	\item  Take a copy $M^3_*$ of $M_*$ and attach one of the boundary components  of $M^3_*$ to that of $M_{2*}$ to form a new surface $M_{3*} = M^3_* \cup  M_{2*}$ in such way that the center curve of  $M_{3*} $ is just a circular arc.
	Also,  $C^3$ smoothness of  $M_{3*}$ is guaranteed by \ref{i3}. Repeat this step until $i$-th iteration and update  $M_{2*}$ by $M_{(i-1)*}$ for each iteration. It will form a new $C^3$ surface $M_{i*}$  which eventually consists of $i$ copies of
	$M_*$. By choosing an appropriate sufficiently small $|\kappa| \in (0, \delta_0)$, after $n$-th iteration, we will have a closed $C^3$ surface $M$ of genus $1$,  composed of $n$ copies of $M_*$,  and the center curve of $M$ is just a whole circle. We notice that the resulting surface $M$ is symmetric with respect to each hyperplane containing a component of the boundary of each $M_\kappa$.
\end{enumerate}

\setcounter{equation}{0}
\setcounter{theorem}{0}
\section{Stability}
In this section, our goal is to present the full proof of Theorem \ref{mt}. Consider the eigenvalue problem for a linearized problem at a stationary solution $U_g$ of \eqref{rdn},
\begin{equation}
\label{egval}
\begin{cases}
 \Delta_{g} q+f^\prime(U_g) q=-\lambda q &\mbox{ in }\ D,\\ 
\frac {\partial q}{\partial\nu}=0 &\mbox{ on }\  \partial D.
\end{cases}
\end{equation}
The principal eigenvalue is characterized by Rayleigh quotient 
\begin{equation} \label{ev}
\lambda_1=  \inf_{\substack{q \neq 0\\q \in H^1(D)}}  \dfrac{ \int\limits_D \left( |\nabla_{g} q|^2 - f^\prime (U_g)q^2 \right) dV_g}{\int\limits _D q^2 dV_g}.
\end{equation}
Let $\phi$ be the eigenfunction of \eqref{rdn} that corresponds to the principal eigenvalue $\lambda_1$ satisfying
\begin{equation}
\label{egval*}
\begin{cases}
 \Delta_{g} \phi+f^\prime(U_g) \phi=-\lambda_1 \phi  &\mbox{ in  } D,\\
 \Vert\phi\Vert_{L^2(D)}=1  ,\quad \phi> 0 & \mbox{ in  } D, \mbox{ and } \frac {\partial \phi}{\partial\nu}=0  \mbox{ on }  \partial D.
 \end{cases}
\end{equation}

It is well known that the principal eigenvalue determines the stability of a stationary solution. The negative principal eigenvalue will lead to the instability of stationary solution $U_g$ while the positive one will assure the stability of $U_g$ (see \cite{eigen} for instance). If $\lambda_1=0$, then the stability of $U_g$ is undetermined.
\subsection{Existence of Patterns }\label{ss}

\begin{theorem}\label{main}
	Let $U_g$  be the patterns of \eqref{rdn} on $D$ given by Theorem \ref{main1}. Then, there exist $\kappa_0 \in (0, \delta_0)$ and a family of patterns $U_\kappa$ of \eqref{rdn}  on $M_{\kappa}$ for all $ | \kappa| \in (0, \kappa_0)$.
\end{theorem}
Let us start with a lemma concerning the existence of a stationary solution of \eqref{rdn} on $M_{\kappa}$ for sufficiently small $|\kappa| > 0$. Functions on $D$ or on $M_\kappa$ are considered as those on $I=[0,l]\times S^1$ given in \eqref{surev}. Conversely, functions on $I$ are considered as those on $D$ or $M_\kappa$. Therefore, we may deal with functions on $I$ instead of those on $D$ or on $M_\kappa$.
\begin{lemma}\label{exist}
	Let $0< \alpha <1$. There exists $\delta_1\in (0,\delta_0)$ such that for every $|\kappa| \in (0,\delta_1)$ a stationary solution $U_\kappa$ of \eqref{rdn} on $M_\kappa$ and $\epsilon_1(\kappa)$ with $\lim\limits_{\kappa \rightarrow 0} \epsilon_1(\kappa)=0$ exist and satisfy
	\begin{equation*}
	\Vert U_\kappa - U_g\Vert_{\mathcal{C}^{2,\alpha}(I)} < \epsilon_1(\kappa)\, \mbox{  if  }\, |\kappa| \in(0, \delta_1).
	\end{equation*}
\end{lemma}
\begin{proof}
	Let $X$ be a closed linear subspace of $C^{2,\alpha}(I)$ with $I =[ 0, l] \times \mathcal{S}^1$ given by
	\begin{equation*}
	X=\{ \upsilon \in C^{2,\alpha}(I)  \, : \, \frac{\partial\upsilon}{\partial s}=0 \mbox{ at } s=0,s=l \}.
	\end{equation*}
	Define for  $(\kappa , \upsilon)$ in $ [-\delta_0, \delta_0]\times X$ the mapping $F : [-\delta_0, \delta_0]\times X \rightarrow C^\alpha(I)$ by
	\begin{equation*}
	F(\kappa,\upsilon)= 
	\Delta_{g^\kappa}(U_g+\upsilon)+f(U_g+ \upsilon).
	\end{equation*}
	 We notice that $F(0,0)=0$ and $F$ is of class $C^1$ in a neighbourhood of $(0,0)$. 
	
The partial Fr\'echet derivative of mapping $F(\kappa, \upsilon)$ with respect to $\upsilon$ at $(0,0)$ is given by
	\begin{equation*}
	\dfrac{\partial F}{\partial \upsilon}(0, 0)q=\Delta_g q+f^\prime(U_g)q \quad \textrm{for } q \in X.
	\end{equation*}
	Let us show that operator $\dfrac{\partial F}{\partial \upsilon}(0, 0)$ is invertible. Indeed, for each $h \in C^\alpha(I)$, we consider the following boundary value problem for $q$ :
	\begin{equation}
	\label{eqh}
	\begin{cases}
	 \Delta_{g} q+f^\prime(U_g) q=h  &\mbox{ in }D,
	\\
	\dfrac{\partial q}{\partial \nu}=0  &\mbox{ on }\partial D. 
	\end{cases}
	\end{equation}
Since $\lambda_1 >0$, for every $h \in C^\alpha(I)$, by the standard theory of elliptic partial differential equations of second order, we see that there exist a unique solution $q\in X$ of \eqref{eqh} and a constant $C>0$ independent of $q$ and $h$ satisfying 
\begin{equation*}
	\left \Vert\dfrac{\partial F}{\partial \upsilon}^{-1}(0, 0)h\right\Vert_{C^{2,\alpha}(I)} =\Vert q\Vert_{C^{2,\alpha}(I)} \leq C 	\Vert h\Vert_{C^{\alpha}(I)}.
\end{equation*}
	Thus by the implicit function theorem (see \cite[Theorema 2.7.2, p. 34]{nirenberg} or \cite[Theorema 15.1, p. 148]{Deimling}), there exist a number $\delta_1\in (0,\delta_0)$ with $ \mathcal{N} =(-\delta_1,\delta_1)$ and a unique  continuous map  $\rho: \mathcal{N} \rightarrow X $ such that $\rho(0)=0$ and for every $\kappa \in \mathcal{N}$
	\begin{equation*}
	F(\kappa, \rho(\kappa))=\Delta_{g^\kappa} U_\kappa+f(U_\kappa)=0,\\
	\end{equation*}
	where we set $U_\kappa=U_g +\rho(\kappa)$. Hence, the result follows. 
\end{proof}

Next, we will determine the stability of the stationary solution $U_\kappa$ of \eqref{rdn} on $M_\kappa$ by using the principal eigenvalue. 	Let   $\lambda_1^\kappa$  be the principal eigenvalue  with the corresponding eigenfunction $\phi^\kappa \in H^1(I)$ of  \eqref{rdn} on $M_{\kappa}$ which satisfies
\begin{equation}\label{prin}
\begin{cases}
\Delta_{g^\kappa}  \phi^\kappa + f^\prime (U_\kappa) \phi^\kappa = -\lambda_1^\kappa \phi^\kappa &\mbox{ in  }M_\kappa,\\ 
\Vert\phi^\kappa\Vert_{L^2(M_{\kappa})}=1  ,\  \phi^\kappa > 0 &\mbox{ in  }M_\kappa,  \mbox{ and } \frac {\partial \phi^\kappa}{\partial\nu}=0  \mbox{ on }  \partial M_\kappa, 
\end{cases}
\end{equation}
where
\begin{equation} \label{ev1}
\lambda_1^\kappa=  \inf_{\substack{q \neq 0\\q \in H^1(M_\kappa)}}  \dfrac{ \int\limits_{M_\kappa}\left( |\nabla_{g^\kappa} q|^2 - f^\prime (U_{g^\kappa})q^2 \right) dV_{g^\kappa}}{\int\limits _{M_\kappa} q^2 dV_{g^\kappa}}.
\end{equation} 

The existence of eigenfunction  $\phi^\kappa$ of \eqref{prin} that corresponds  to $\lambda_1^\kappa$ can be confirmed in the  way similar to that  in  \cite[Chapter 8.12]{buku}. 
\begin{lemma}\label{int}
	For every $| \kappa|  \in (0,\delta_1)$, there exists $\epsilon_2(\kappa)$ with $\lim\limits_{\kappa \rightarrow 0}\epsilon_2(\kappa)=0$ such that
	\begin{enumerate}[label=(\roman*),font=\upshape]
		\item 
		$\| f^\prime (U_\kappa)- f^\prime (U_g)\|_{\infty} \leq \epsilon_2 (\kappa)$.
		\item $(1 -\epsilon_2(\kappa))dV_{g}  \leq  dV_{g^\kappa}\leq  (1 +\epsilon_2(\kappa))dV_{g}$.
	 	\item $(1- \epsilon_2(\kappa) ) |\nabla_{g}\phi^\kappa|^2 \leq |\nabla_{g^\kappa}\phi^\kappa|^2 \leq (1+ \epsilon_2(\kappa) ) |\nabla_{g}\phi^\kappa|^2 $.
\end{enumerate}
\end{lemma}
\begin{proof}
			Lemma \ref{exist} and the continuity of $f^\prime$ yield  assertion (i), for some $\epsilon_2(\kappa)>0$ with $\lim\limits_{\kappa \rightarrow 0}\epsilon_2(\kappa)=0$.
			Next, represent $dV_{g^\kappa} $ by using Taylor expansion with respect to $\kappa$ at $\kappa=0$. Since $\dfrac{\partial \left(\sqrt{|g^\kappa|}\right)}{\partial \kappa}$ is continuous on $I \times [-\delta_1,\delta_1]$, there exists a constant $m_1 > 0$ satisfying $\left |\dfrac{\partial \left(\sqrt{|g^\kappa|}\right)}{\partial \kappa}\right| \leq m_1$. Then, we obtain
			\begin{equation*}\label{exp1}
			(1 -\epsilon_2(\kappa))dV_{g}  \leq \left(\sqrt{|g|} -  m_1|\kappa|\right)ds d\theta \leq dV_{g^\kappa}\leq \left(\sqrt{|g|} +  m_1|\kappa|\right)ds d\theta \leq  (1 +\epsilon_2(\kappa))dV_{g},
			\end{equation*}
			if we choose  $\epsilon_2(\kappa)\geq \dfrac{m_1 |\kappa|}{\min_D \sqrt{|g|}}$. Thus, this chain of inequalities gives assertion (ii). 
			
It remains to show assertion (iii). We have that
\begin{equation*}
|\nabla_{g^\kappa} \phi^\kappa|^2 = \dfrac{1}{(\Psi^\prime(s))^2+(\kappa \Psi(s)\cos \theta-1)^2} (\partial_s \phi^\kappa)^2+\dfrac{1}{\Psi^2}(\partial_\theta \phi^\kappa)^2.
\end{equation*}
For $(g^\kappa_{11})^{-1}= \dfrac{1}{(\Psi^\prime(s))^2+(\kappa \Psi(s)\cos \theta-1)^2}$, $\dfrac{\partial(g^\kappa_{11})^{-1}}{\partial \kappa}$  is continuous on  $I \times [-\delta_1,\delta_1]$. Hence, there exists a constant $m_2 > 0$  satisfying  $\left |\dfrac{\partial(g^\kappa_{11})^{-1}}{\partial \kappa}\right| \leq m_2$. Also, there exists a constant $C>0$ such that $(\partial_s \phi^\kappa)^2 \leq C |\nabla_{g} \phi^\kappa|^2$. Then, we get
\begin{equation*}
\Big| |\nabla_{g^\kappa} \phi^\kappa|^2- |\nabla_{g} \phi^\kappa|^2 \Big| \leq m_2|\kappa|  |\partial_s \phi^\kappa|^2 \leq m_2|\kappa|C|\nabla_{g} \phi^\kappa|^2.
\end{equation*}
This shows (iii), if we choose $\epsilon_2(\kappa)\geq \max \Bigg \lbrace \dfrac{m_1 |\kappa|}{\min_D \sqrt{|g|}}, m_2|\kappa|C\Bigg \rbrace$.
\end{proof}

\begin{lemma}\label{bounded}
	 There exists a constant $C^* >0$ such that  if $ |\kappa| <  \delta_1$ then
	 $$
	 \int\limits_{M_\kappa}|\nabla_{g^\kappa} \phi^\kappa|^2 dV_{g^\kappa} \leq C^*.
	 $$
\end{lemma}
\begin{proof}
	We have from \eqref{prin} that
	\begin{align}\label{apa}
	\int\limits_{M_{\kappa}}  |\nabla_{g^\kappa} \phi^\kappa|^2  dV_{g^\kappa}=\lambda_1^\kappa + \int\limits_{M_{\kappa}} f^\prime (U_\kappa) (\phi^\kappa)^2 dV_{g^\kappa}.
	\end{align} Since $\lambda_1^\kappa$ is the principal eigenvalue (see \eqref{ev1}), for every $q \in H^1(I)$
	\begin{align*}
	\lambda_1^\kappa\leq \dfrac{\int\limits_{M_{\kappa}} \left( |\nabla_{g^\kappa} q|^2  - f^\prime (U_\kappa) q^2\right)  dV_{g^\kappa} }{ \int\limits_{M_{\kappa}} q^2 dV_{g^\kappa}}.
	\end{align*}
	Take $q \equiv 1$ and use (i) and (ii) to obtain
	\begin{align*}
	\lambda_1^\kappa \leq \dfrac{-\int\limits_{M_{\kappa}} f^\prime (U_\kappa)  dV_{g^\kappa}}{  \int\limits_{M_{\kappa}}  dV_{g^\kappa}}&=\dfrac{-\int\limits_{M_{\kappa}} \left(f^\prime (U_\kappa)-f^\prime(U_g) \right) dV_{g^\kappa}- \int\limits_{M_{\kappa}} f^\prime (U_g) dV_{g^\kappa}}{  \int\limits_{M_{\kappa}} dV_{g^\kappa}}\\[0.25cm]
	&\leq \epsilon_2(\kappa)+ \dfrac{\Bigl|\int\limits_{M_{\kappa}} f^\prime (U_g)  dV_{g^\kappa }\Bigr|}{  \int\limits_{M_{\kappa}} dV_{g^\kappa}} \\[0.25cm]
	&\leq \epsilon_2(\kappa)+\dfrac{(1+\epsilon_2(\kappa))\int\limits_{D} |f^\prime (U_g)|  dV_{g}}{(1-\epsilon_2(\kappa))  \int\limits_{D}  dV_{g}} \\[0.25cm]
	&\leq \epsilon_2(\kappa)+\dfrac{1+\epsilon_2(\kappa)}{1-\epsilon_2(\kappa)  }\max_D |f^\prime(U_g)|.
	\end{align*}
	Hence, $\lambda_1^\kappa$ is bounded from above.
	
	Next, we  estimate $ \int\limits_{M_{\kappa}} f^\prime (U_\kappa) (\phi^\kappa)^2 dV_{g^\kappa}$. 
	\begin{align*}
	\int\limits_{M_{\kappa}} f^\prime (U_\kappa) (\phi^\kappa)^2 dV_{g^\kappa} &= \int\limits_{M_{\kappa}} \left(f^\prime (U_\kappa)-f^\prime(U_g)\right) (\phi^\kappa)^2 dV_{g^\kappa} +  \int\limits_{M_{\kappa}}f^\prime (U_g) (\phi^\kappa)^2 dV_{g^\kappa}  \\[0.25cm]
	&\leq \epsilon_2(\kappa) \int\limits_{M_{\kappa}} (\phi^\kappa)^2 dV_{g^\kappa}+ \int\limits_{M_{\kappa}} |f^\prime (U_g)| (\phi^\kappa)^2 dV_{g^\kappa} \\[0.25cm]
	&\leq \epsilon_2(\kappa) + \max_D |f^\prime (U_g)|.
	\end{align*}
	Then, \eqref{apa} yields the conclusion.
\end{proof}

\begin{lemma}\label{uni}
	 $\lambda_1^\kappa \rightarrow \lambda_1 \textrm{ as }  \kappa \rightarrow 0$.
\end{lemma}
\begin{proof}
Since $\lambda_1$ is the principal eigenvalue (see \eqref{ev}), we have
	\begin{align*}
	\lambda_1 \leq \dfrac{\int\limits_D \left( |\nabla_{g} \phi^\kappa|^2  - f^\prime (U_g) (\phi^\kappa)^2\right)  dV_g }{ \int\limits_D (\phi^\kappa)^2 dV_{g}}.
	\end{align*}
	Observe that
	\begin{align}
	&\biggl|\int\limits_D f^\prime (U_g) (\phi^\kappa)^2 dV_{g}- \int\limits_{M_\kappa}f^\prime (U_\kappa) (\phi^\kappa)^2 dV_{g^\kappa}\biggr| \nonumber \\[0.25cm]
	& \leq\biggl|\int\limits_D \left(f^\prime (U_g)-f^\prime(U_\kappa)\right) (\phi^\kappa)^2 dV_{g} \biggr| 
	+\biggl|\int\limits_{D} f^\prime (U_\kappa) (\phi^\kappa)^2dV_{g}-\int\limits_{M_\kappa} f^\prime (U_\kappa) (\phi^\kappa)^2dV_{g^\kappa}\biggr|.\label{if}
	\end{align}
Then, by using (i) and (ii) in Lemma \ref{int} we obtain
\begin{align}
\biggl| \int\limits_D \left(f^\prime (U_g)-f^\prime(U_\kappa)\right)(\phi^\kappa)^2 dV_{g}\biggr| &\leq \epsilon_2(\kappa) \int\limits_{D} (\phi^\kappa)^2 dV_{g} \nonumber \\[0.25cm] &\leq  \dfrac{\epsilon_2(\kappa)}{1-\epsilon_2(\kappa)}\int\limits_{M_\kappa} (\phi^\kappa)^2 dV_{g^\kappa}=\dfrac{\epsilon_2(\kappa)}{1-\epsilon_2(\kappa)}.\label{e1}
\end{align}
Since $\left| dV_{g^\kappa}-dV_{g}\right| \leq \dfrac{\epsilon_2(\kappa)}{1-\epsilon_2(\kappa)} dV_{g^\kappa} $ by (ii),
\begin{flalign}\label{e2}
\biggl|\int\limits_{D} f^\prime (U_\kappa) (\phi^\kappa)^2dV_{g}-\int\limits_{M_\kappa} f^\prime (U_\kappa) (\phi^\kappa)^2dV_{g^\kappa}\biggr|&\leq\dfrac{\epsilon_2(\kappa)}{1-\epsilon_2(\kappa)} \int\limits_{M_\kappa} |f^\prime(U_\kappa)|(\phi^\kappa)^2dV_{g^\kappa}& \nonumber\\[0.25cm]
&\leq\dfrac{\epsilon_2(\kappa)}{1-\epsilon_2(\kappa)} \left(\epsilon_2(\kappa)+ \| f^\prime (U_g)\|_{\infty}\right),&
\end{flalign}
where we used (i) in the last inequality. Thus, combining \eqref{e1} and  \eqref{e2} with \eqref{if} yields  \begin{equation}\label{f1}
\biggl|\int\limits_D f^\prime (U_g) (\phi^\kappa)^2 dV_{g}- \int\limits_{M_\kappa}f^\prime (U_\kappa) (\phi^\kappa)^2 dV_{g^\kappa}\biggr| \leq \dfrac{\epsilon_2(\kappa)}{1-\epsilon_2(\kappa)} \left(1+\epsilon_2(\kappa)+\| f^\prime (U_g)\|_{\infty}\right).
\end{equation}
Observe that
\begin{align} \label{f2}
&\biggl| \int\limits_{D} |\nabla_{g} \phi^\kappa|^2 dV_g - \int\limits_{M_\kappa}|\nabla_{g^\kappa} \phi^\kappa|^2dV_{g^\kappa}\biggr| \nonumber \\[0.25cm] &\leq \biggl| \int\limits_{D} |\nabla_{g} \phi^\kappa|^2 dV_g - \int\limits_{M_\kappa}|\nabla_{g} \phi^\kappa|^2dV_{g^\kappa}\biggr| 
+\biggl| \int\limits_{M_\kappa} |\nabla_{g} \phi^\kappa|^2 dV_{g^\kappa} - \int\limits_{M_\kappa}|\nabla_{g^\kappa} \phi^\kappa|^2 dV_{g^\kappa}\biggr| .
\end{align}
Then, by using (ii) and (iii) in Lemma \ref{int}, and Lemma \ref{bounded} , we obtain
\begin{flalign}
\biggl| \int\limits_{D} |\nabla_{g} \phi^\kappa|^2 dV_g - \int\limits_{M_\kappa}|\nabla_{g} \phi^\kappa|^2dV_{g^\kappa}\biggr| &\leq \dfrac{\epsilon_2(\kappa)}{1-\epsilon_2(\kappa)} \int\limits_{M_{\kappa}} |\nabla_{g} \phi^\kappa|^2 dV_{g^\kappa} \nonumber& \\[0.25cm] 
&\leq \dfrac{\epsilon_2(\kappa)}{(1-\epsilon_2(\kappa))^2} \int\limits_{M_{\kappa}} |\nabla_{g^\kappa} \phi^\kappa|^2 dV_{g^\kappa}&  \nonumber\\[0.25cm]
&\leq \dfrac{\epsilon_2(\kappa)}{(1-\epsilon_2(\kappa))^2}C^*. \label{e3}& 
\end{flalign}
By using (iii), we have
\begin{flalign*}
\Big| |\nabla_{g} \phi^\kappa|^2 -|\nabla_{g^\kappa} \phi^\kappa|^2 \Big| \leq \epsilon_2(\kappa) |\nabla_{g} \phi^\kappa|^2 \leq \dfrac{\epsilon_2(\kappa)}{1-\epsilon_2(\kappa)} |\nabla_{g^\kappa} \phi^\kappa|^2.
\end{flalign*}
Hence, we obtain from Lemma \ref{bounded} that 
\begin{align}
\label{e4}\biggl| \int\limits_{M_\kappa} |\nabla_{g} \phi^\kappa|^2 dV_{g^\kappa} - \int\limits_{M_\kappa}|\nabla_{g^\kappa} \phi^\kappa|^2 dV_{g^\kappa}\biggr| 
\leq \dfrac{\epsilon_2(\kappa)}{1-\epsilon_2(\kappa)}C^*.
\end{align}
Thus, combining \eqref{e3} and \eqref{e4} with \eqref{f2} yields
\begin{equation}\label{f3}
\biggl| \int\limits_{D} |\nabla_{g} \phi^\kappa|^2 dV_g - \int\limits_{M_\kappa}|\nabla_{g^\kappa} \phi^\kappa|^2dV_{g^\kappa}\biggr|  \leq \dfrac{\epsilon_2(\kappa)C^*(2-\epsilon_2(\kappa))}{(1-\epsilon_2(\kappa))^2}.
\end{equation}
By (ii) in Lemma \ref{int}, the following inequality holds:
\begin{align}\label{e5}
 \biggl| \int\limits_D (\phi^\kappa)^2 dV_{g}-1 \biggr| &=\biggl| \int\limits_D (\phi^\kappa)^2 dV_{g}-\int\limits_{M_\kappa} (\phi^\kappa)^2 dV_{g^\kappa} \biggr| \nonumber \\[0.25cm]
 &\leq \dfrac{\epsilon_2(\kappa)}{1-\epsilon_2(\kappa)} \int\limits_{M_\kappa} (\phi^\kappa)^2 dV_{g^\kappa} =\dfrac{\epsilon_2(\kappa)}{1-\epsilon_2(\kappa)}.
\end{align}
By using \eqref{f1}, \eqref{f3}  and  \eqref{e5}, we conclude that there exists $\epsilon_3(\kappa) > 0$ with $\lim\limits_{\kappa \rightarrow 0} \epsilon_3(\kappa)=0$ such that
\begin{equation}\label{l1}
 \lambda_1 \leq \lambda_1^\kappa +\epsilon_3(\kappa).
\end{equation}

Since $\lambda_1^\kappa$ is the principal eigenvalue  (see \eqref{prin}), we have
\begin{align*}
\lambda_1^\kappa \leq \dfrac{\int\limits_{M_\kappa} \left( |\nabla_{g} \phi|^2  - f^\prime (U_\kappa) (\phi)^2\right)  dV_{g^\kappa }}{ \int\limits\limits_{ M_\kappa} (\phi)^2 dV_{g^\kappa}}.
\end{align*}
By the same argument as above,  we can show that there exists $\epsilon_4(\kappa
)>0$ with $\lim\limits_{\kappa \rightarrow 0} \epsilon_4(\kappa)=0$ satisfying
\begin{equation}\label{l2}
\lambda_1^\kappa \leq \lambda_1 + \epsilon_4(\kappa).
\end{equation}
Combining \eqref{l1} and \eqref{l2} yields the conclusion.
\end{proof}

Now we are in position to prove Theorem \ref{main}.
\begin{proof}[Proof of  Theorem \ref{main}]
	Theorem \ref{main} is a consequence of Lemma \ref{exist} and Lemma \ref{uni}. According to Lemma \ref{exist}, the stationary solution $U_\kappa$ exists in the neighborhood of $\kappa=0$. By Lemma \ref{uni}, for sufficiently small  $|\kappa|$ we have that $\lambda_1^\kappa > 0$ which leads to the stability of $U_\kappa$.
\end{proof}

\subsection{Proof of Theorem \ref{mt}}
Theorem \ref{main} yields that for $|\kappa| \in (0,\kappa_0)$  there exists a pattern $U_\kappa$ of \eqref{rdn}  on $M_{\kappa}$.  Moreover, by the construction in section 3,  there exist closed surfaces $M$ having genus $1$  for every appropriate  $|\kappa| \in (0,\kappa_0)$.  For such a $M$, let us define the stationary solution $U$ of \eqref{eq1} by
\begin{equation*}
U(x)=U_\kappa(x) \quad \textrm{for } x \in M_\kappa,
\end{equation*}
since each $U_\kappa$ satisfies the Neumann boundary condition on each $\partial M_\kappa$. We notice that $U$ is also symmetric with respect to each component of each $\partial M_\kappa$
\begin{lemma}\label{main2}
 $U$ is stable on $M$. 
\end{lemma}

\begin{proof}
Choose $\epsilon >0$. Then, since $U_\kappa$ is stable on $M_\kappa$, there exists $\delta > 0$ such that for each initial data $u_\kappa(\cdot,0) \in L^\infty(M_{\kappa})$ with 
\begin{equation*}
\Vert u_\kappa(\cdot,0) - U_\kappa \Vert_{\infty, M_{\kappa}} < 2 \delta,
\end{equation*}
the solution $u_\kappa \in C^{2,1}(M_{g^\kappa} \times (0, \infty)) \cap C([0,\infty], L^1(M_\kappa))$ of \eqref{rdn} satisfies
\begin{equation}\label{stab}
\Vert u_\kappa(\cdot,t) - U_\kappa\Vert_{\infty, M_{\kappa}} < \epsilon\ \mbox{ for every } t > 0.
\end{equation}

Let $u_0 \in L^\infty(M)$  satisfies $
\Vert u_0 - U \Vert_{\infty, M} < \delta.
$
Set $u^+_0 :=U+ \delta $ and $u^-_0 :=U-\delta $ on $M$ so that the following inequalities hold: 
\begin{equation*}
u^-_0 < u_0 < u^+_0 \quad \textrm{on } M_\kappa, 
\end{equation*}
\begin{equation*}
\Vert u^+_0 - U \Vert_{\infty, M} =\delta<2\delta,
\quad \textrm{and} \quad
\Vert u^-_0 -U \Vert_{\infty, M}=\delta<  2\delta.
\end{equation*}
Let $ u^+,u^- \in  C(M \times [0,\infty)) \cap  C^{2,1}(M \times (0,\infty))$ be the solutions of \eqref{eq1} on $M$  with initial data $u^+_0, u^-_0$, respectively. Since both  $u^+_0$ and $u^-_0$ are symmetric with respect to each component of each $\partial M_\kappa$,  both $u^+$ and $u^-$ are also symmetric in the same manner for every $t>0$. Hence both $u^+$ and $u^-$ satisfy \eqref{rdn} where $D$ is replaced with each $M_\kappa$. Thus, by \eqref{stab},
\begin{align}\label{stab2}
\Vert u^+ (\cdot,t) - U \Vert_{\infty, M} <  \epsilon \  \mbox{ and } \ 
\Vert u^-(\cdot,t) -U \Vert_{\infty, M} <  \epsilon \ \mbox{ for every } t > 0.
\end{align}
 Since $u^-_0 < u_0 < u^+_0$ on $M$,   by the comparison principle there exists a unique solution $u  \in C^{2,1}(M \times (0, \infty)) \cap C([0,\infty), L^1(M))$ of \eqref{eq1} on $M$  with initial data $u_0$ and it satisfies
 $u^-< u < u^+$ on $M \times [0,\infty)$. Therefore it follows from \eqref{stab2}  that
\begin{equation*}
\Vert u(\cdot, t) - U\Vert_{\infty, M} < \epsilon\ \mbox{ for every } t > 0,
\end{equation*}
which shows that $U$ is stable  on $M$.  \end{proof}
               
   Recall that $\dfrac{\partial U_\kappa}{\partial \nu}=0$ on $\partial M_\kappa$. Then there exist at least two critical points of $U$  on each component of each $\partial M_{\kappa}$.   Therefore we infer that $U$ has at least $4n$ critical points with $ n >  \dfrac{\pi}{\gamma\kappa_0} $ where  $\gamma$ is the length of $C(\kappa)$, since $U$ is symmetric with respect to each hyperplane containing each component of each $\partial M_\kappa$.  This completes the proof of Theorem \ref{mt}.
	 
\end{document}